\documentclass{amsart}
\usepackage{amssymb}
\usepackage{amsmath}
\usepackage{a4}

\newtheorem{theorem}{Theorem}[section]
\newtheorem{lemma}[theorem]{Lemma}
\newtheorem{remark}[theorem]{Remark}
\newtheorem{definition}[theorem]{Definition}

\def\qedbox{\hbox{$\rlap{$\sqcap$}\sqcup$}}
\makeatletter
  
  \@addtoreset{equation}{section}
 \makeatother

\begin{document}

\title[ Commuting curvature operators]
{ Examples of signature $(2,2)$ manifolds with commuting curvature operators}%
\author[Brozos-V\'azquez et. al.]
{M. Brozos-V\'azquez,
  E. Garc\'{\i}a--R\'{\i}o, P. Gilkey
 and R. V\'azquez-Lorenzo}
\begin{address}{B-V, G-R, V-L: Department of Geometry and Topology, Faculty of Mathematics,
University of Santiago de Compostela, Santiago de Compostela 15782, Spain.}\end{address}
\begin{email}{mbrozos@usc.es, xtedugr@usc.es, ravazlor@usc.es}\end{email}
\begin{address}{G: Mathematics Department, University of Oregon, Eugene Or 97403, USA}
\end{address}\begin{email}{gilkey@uoregon.edu}\end{email}

\begin{abstract}
 We exhibit Walker manifolds of signature $(2,2)$ with various commutativity properties for the
Ricci operator, the skew-symmetric curvature operator, and the Jacobi operator. If the Walker metric is
a Riemannian extension of an underlying affine structure $\mathcal{A}$, these properties are related to
the Ricci tensor of $\mathcal{A}$.
\end{abstract}

\maketitle

\section{Introduction}\label{sect-1}

Let $\mathcal{M}:=(M,g)$ be a pseudo-Riemannian manifold of
signature $(p,q)$, let $\nabla$ be the associated Levi-Civita
connection, and let
$\mathcal{R}(x,y):=\nabla_x\nabla_y-\nabla_y\nabla_x-\nabla_{[x,y]}$
be the curvature operator. With our sign convention, the {\it Jacobi
operator} is given by
$\mathcal{J}(x):y\rightarrow\mathcal{R}(y,x)x$. Let $\mathcal{\rho}$
be the associated Ricci operator; $g(\rho
x,x)=\operatorname{Tr}\{\mathcal{J}(x)\}$. We shall study relations
between algebraic properties of the curvature operator and the
underlying geometry of the manifold. Commutativity conditions
of curvature operators have been considered extensively in the study
of submanifolds (see for example \cite{Ortega, Ts05}), hence it is natural to look at them from a broader
intrinsic point of view.

\begin{definition}\label{defn-1.1}
$\mathcal{M}$ is said to be:\begin{enumerate}
\item {\it Einstein} if $\rho$ is a scalar multiple of the identity.
\item {\it  Jacobi--Ricci commuting} if
$\mathcal{J}(x)\rho=\rho\mathcal{J}(x)$ $\forall$ $x$.
\item {\it  curvature--Ricci commuting} if $\mathcal{R}(x,y)\rho=\rho\mathcal{R}(x,y)$ $\forall$
$x,y$.
\item {\it  Jacobi--Jacobi commuting} if
$\mathcal{J}(x)\mathcal{J}(y)=\mathcal{J}(y)\mathcal{J}(x)$ $\forall$ $x,y$.
\item {\it  curvature--Jacobi commuting} if
$\mathcal{J}(x)\mathcal{R}(y,z)=\mathcal{R}(y,z)\mathcal{J}(x)$ $\forall$ $x,y,z$.
\item {\it  curvature--curvature commuting} if
$\mathcal{R}(w,x)\mathcal{R}(y,z)=\mathcal{R}(y,z)\mathcal{R}(w,x)$\newline $\forall$ $w,x,y,z$.
\end{enumerate}\end{definition}

Commutativity properties of the skew-symmetric curvature operator and of the Jacobi operator were first
studied in the Riemannian setting by Tsankov
\cite{Ts05}. He showed that if $\mathcal{M}$ is a hypersurface in
$\mathbb{R}^{m+1}$ with
$\mathcal{J}(x)\mathcal{J}(y)=\mathcal{J}(y)\mathcal{J}(x)$ for
all $x\bot y$, then necessarily $\mathcal{M}$ had constant
sectional curvature; this result was subsequently extended to the
general Riemannian context in \cite{BGa05} and additional results
obtained in the general pseudo-Riemannian setting in
\cite{BGb06,BGNb05,IZV}. Tsankov also derived results relating to
hypersurfaces where
$\mathcal{R}(w,x)\mathcal{R}(y,z)=\mathcal{R}(y,z)\mathcal{R}(w,x)$.
Videv studied manifolds where $\rho
\mathcal{J}(x)=\mathcal{J}(x)\rho$ for all $x$. Many of these investigations were
originally suggested by Stanilov \cite{S1,S2}.

The conditions in Definition \ref{defn-1.1} have also been described elsewhere in the literature as ``Jacobi--Videv'',
``skew--Videv'', ``Jacobi--Tsankov'', ``mixed--Tsankov'', and ``skew--Tsankov'', respectively
and the general field of investigation of such conditions is often referred to as
Stanilov--Tsankov--Videv theory.
The  curvature--Ricci commuting condition has also been denoted as
``Ricci semi-symmetric''; it is a generalization of the semi-symmetric
condition (see \cite{AD} and the references therein).
Semi-symmetric manifolds of conullity two are curvature--curvature commuting \cite{BGb07, BKV}.
We have chosen
to change the notation from that employed previously to put   these conditions in parallel as much as possible.

In this paper, we shall exhibit families of manifolds having some, but not necessarily all, of
these properties in order to obtain insight into relationships between these concepts and the
underlying geometry. We shall work with a restricted class of Walker manifolds of signature $(2,2)$;
this class is both sufficiently rich to offer a significant number of examples and sufficiently
restricted to permit a relatively complete analysis.
We have not attempted to obtain the most general possible
classification results for Walker signature $(2,2)$ manifolds as our experience in similar related
problems is that these tend to be excessively technical; for example, there is as yet no classification
of Einstein Walker signature $(2,2)$ manifolds and there is as yet no classification of anti-self-dual
Walker signature $(2,2)$ manifolds. As the family we shall examine has been studied extensively in other contexts
\cite{BGVa05,BGGV06}, we can also relate curvature commutativity properties for these manifolds to other properties such as Einstein,
self-dual, anti-self-dual and Osserman.

Walker \cite{W} studied pseudo-Riemannian manifolds with a parallel field of null planes and derived a canonical
form. Lorentzian Walker metrics have been studied extensively in the physics
literature since they constitute the background metric of the pp-wave models (\cite{ABRS, K, KNP, MW} to list a
few of the many possible references; the literature is a vast one); a pp-wave spacetime admits a covariantly
constant null vector field $U$ and therefore it is trivially recurrent (i.e., $\nabla U=\omega\otimes U$ for some one-form $\omega$).
Lorentzian Walker metrics present many specific features both from the physical and geometric viewpoints  \cite{Calvaruso, CGRVA, Leistner, PPCM}. We also refer to related work of Hall \cite{H91} and of Hall and
 da Costa \cite{HS81} for generalized Lorentzian Walker manifolds (i.e. for spacetimes admitting a
nonzero vector field $n^a$ satisfying $R_{ijkl}n^l=0$ or admitting a rank $2$-symmetric or
anti-symmetric tensor $H_{ab}$ with $\nabla H=0$).

One says that a pseudo-Riemannian manifold $\mathcal{M}$ of signature $(2,2)$ is a {\it Walker
manifold} if it admits a parallel totally isotropic $2$-plane field; see \cite{CGM05, M05} for further
details. Such a manifold is locally isometric to an example of the following form:
$\mathcal{M}:=(\mathcal{O},g)$ where $\mathcal{O}$ is an open subset of $\mathbb{R}^4$ and where the metric is
given by:
\begin{equation}\label{eqn-1.a}
\begin{array}{l}
g(\partial_{x_1},\partial_{x_3})=g(\partial_{x_2},\partial_{x_4})=1,\\
g(\partial_{x_i},\partial_{x_j})=g_{ij}(x_1,x_2,x_3,x_4)\text{ for }i,j=3,4;
\end{array}\end{equation}
here $(x_1,x_2,x_3,x_4)$ are coordinates on $\mathbb{R}^4$.  In this paper, we shall examine the
concepts of Definition \ref{defn-1.1} for a restricted category of signature $(2,2)$ Walker
metrics where we set $g_{33}=g_{44}=0$:
\begin{equation}\label{eqn-1.b}
g(\partial_{x_1},\partial_{x_3})=g(\partial_{x_2},\partial_{x_4})=1,\quad
g(\partial_{x_3},\partial_{x_4})=g_{34}(x_1,x_2,x_3,x_4)\,.
\end{equation}

Let
$dx_1dx_2dx_3dx_4$ orient $\mathbb{R}^4$.
 The study of self-dual and anti-self-dual metrics is crucial in
Lorentzian geometry, see, for example, \cite{ABP05, HSW}. The same is true in the higher signature
context \cite{BS, DM, DGV2, HSW, LPS}. One says that $\mathcal{M}$ is {\it Osserman} if
the spectrum of the Jacobi operator is constant on the pseudo-sphere of unit spacelike vectors  or,
equivalently, on the pseudo-sphere of unit timelike vectors. The notion of {\it conformally Osserman} is
defined using the conformal Jacobi operator. One has that $\mathcal{M}$ is conformally Osserman
$\Leftrightarrow$ $\mathcal{M}$ is either self-dual or anti-self-dual \cite{BGVa05}.   If
$f=f(x_1,x_2,x_3,x_4)$, let $f_{/i}:=\partial_{x_i}f$ and let $f_{/ij}:=\partial_{x_i}\partial_{x_j}f$.
One has the following surprising result:

\begin{theorem}\label{thm-1.3}
Let the metric be as in Equation (\ref{eqn-1.b}). Then
\begin{enumerate}
\item $\mathcal{M}$ is self-dual $\Leftrightarrow$
$g_{34}=x_1p(x_3,x_4)+x_2q(x_3,x_4)+s(x_3,x_4)$.
\item $\mathcal{M}$ is anti-self-dual $\Leftrightarrow$
$g_{34}=x_1p(x_3,x_4)+x_2q(x_3,x_4)+s(x_3,x_4)$ $+\xi(x_1,x_4)$ $+\eta(x_2,x_3)$ with
$p_{/3}=q_{/4}$ and $g_{34}p_{/3}-x_1p_{/34}-x_2p_{/33}-s_{/34}=0$.
\item The following assertions are equivalent:
\begin{enumerate}
\item $\mathcal{M}$ is Osserman.
\item $\mathcal{M}$ is Einstein.
\item $g_{34}=x_1p(x_3,x_4)+x_2q(x_3,x_4)+s(x_3,x_4)$ where $p^2=2p_{/4}$, $q^2=2q_{/3}$, and
$pq=p_{/3}+q_{/4}$.
\item $\rho=0$.
\end{enumerate}\end{enumerate}
\end{theorem}

We emphasize that it is a crucial feature of these examples that Ricci flat, Einstein, and Osserman are
equivalent conditions; this is not the case, of course, for general Walker metrics of signature $(2,2)$.

The conditions on $p$ and $q$ which are given in Assertion (3c) of Theorem \ref{thm-1.3} will play
an important role in what follows. The following is a useful technical result that will be central
in our discussions:

\begin{lemma}\label{lem-1.4}
Let $\mathcal{O}$ be an open connected subset of $\mathbb{R}^4$. Let
$p,q\in C^\infty(\mathcal{O})$ be functions only of $(x_3,x_4)$. Then the
following conditions are equivalent:
\begin{enumerate}
\item $p^2=2p_{/4}$, $q^2=2q_{/3}$, and $pq=p_{/3}+q_{/4}$.
\item $p^2=2p_{/4}$, $q^2=2q_{/3}$, and $p_{/3}=q_{/4}=\frac12pq$.
\item There exist $(a_0,a_3,a_4)\in\mathbb{R}^3-\{0\}$ so that
$p=-2a_4(a_0+a_3x_3+a_4x_4)^{-1}$ and $q=-2a_3(a_0+a_3x_3+a_4x_4)^{-1}$.
\end{enumerate}
\end{lemma}

Jacobi--Ricci commuting and  curvature--Ricci commuting are equivalent concepts in the context of metrics given by Equation
(\ref{eqn-1.b}).

\begin{theorem}\label{thm-1.5}
Let $\mathcal{M}$ be given by Equation (\ref{eqn-1.b}). The following assertions are equivalent:
\begin{enumerate}
\item $\mathcal{M}$ is
 Jacobi--Ricci commuting.
\item $\mathcal{M}$ is  curvature--Ricci commuting.
\item $g_{34}=x_1p(x_3,x_4)+x_2q(x_3,x_4)+s(x_3,x_4)$ where
$p_{/3}=q_{/4}$.\end{enumerate}
\end{theorem}

Lemma \ref{lem-1.4} shows that the conditions of Theorem \ref{thm-1.3} (3c) are very rigid. On the
other hand, the condition of Theorem \ref{thm-1.5} (3) that $p_{/3}=q_{/4}$ is, of course, nothing
but the condition that $\omega:=p\,dx_4+q\,dx_3$ is a closed $1$-form and thus there are many
examples.

We now turn our attention to Tsankov theory.
A Riemannian or Lorentzian manifold is Jacobi-Jacobi commuting if and only if
it is of constant sectional curvature \cite{BGa05, BGb06}. This is not the case in the
higher signature context. Further observe that any Jacobi-Jacobi commuting metric given by
Equation (\ref{eqn-1.b}) is semi-symmetric since the Jacobi operators are two-step nilpotent
\cite{GaVa}. Also observe that curvature-Ricci and curvature-curvature commuting are equivalent
conditions for metrics (\ref{eqn-1.b}), which is not a general fact (see Theorem \ref{thm-1.10} and Remark \ref{aaa}).

\begin{theorem}\label{thm-1.6}
Let $\mathcal{M}$ be given by Equation (\ref{eqn-1.b}).
\begin{enumerate}
\item The following assertions are equivalent:
\begin{enumerate}
\item $\mathcal{M}$ is  curvature--curvature commuting.
\item
$g_{34}=x_1p(x_3,x_4)+x_2q(x_3,x_4)+s(x_3,x_4)$ where
$p_{/3}=q_{/4}$.
\end{enumerate}
\item  Let $\mathcal{P}:=\operatorname{Span}\{\partial_{x_1},\partial_{x_2}\}$.
 The following assertions are
equivalent:
\begin{enumerate}
\item $\mathcal{R}(x,y)z\in\mathcal{P}$ for all $x,y,z$ and
$\mathcal{R}(x,y)z=0$ if $x$, $y$, or $z$ is in $\mathcal{P}$.
\item $\mathcal{M}$ is  curvature--Jacobi commuting.
\item $\mathcal{M}$ is Jacobi--Jacobi commuting.
\item $\mathcal{J}(x)^2=0$ for all $x$.
\item $\rho=0$.
\item $g_{34}=x_1p(x_3,x_4)+x_2q(x_3,x_4)+s(x_3,x_4)$ where $p^2=2p_{/4}$, $q^2=2q_{/3}$, and
$pq=p_{/3}+q_{/4}$.
\end{enumerate}\end{enumerate}
\end{theorem}

As it is a feature of our analysis that the warping function $g_{34}$ is
necessarily affine, it is worth putting such metrics in a geometrical context.
Let $\nabla$ be a torsion free connection on a smooth manifold $N$; the resulting
structure $\mathcal{A}:=(N,\nabla)$ is said to be an {\it affine manifold}. The associated
{\it Jacobi operator} and {\it Ricci tensor} are defined, respectively, by
$$\mathcal{J}_{\mathcal{A}}(x):y\rightarrow\mathcal{R_{\mathcal{A}}}(y,x)x\quad\text{and}
\quad
\rho_{\mathcal{A}}(x,y):=\operatorname{Tr}\{z\rightarrow\mathcal{R_{\mathcal{A}}}(z,x)y\}\,.$$
We say
$\mathcal{A}$ is {\it affine Osserman} if $\mathcal{J}_{\mathcal{A}}(x)$ is nilpotent for all
tangent vectors $x$, i.e. if $\operatorname{Spec}\{\mathcal{J}_{\mathcal{A}}(x)\}=\{0\}$ for all tangent vectors.
If the affine structure arises from a pseudo-Riemannian structure, i.e. if $\nabla$ is the Levi-Civita
connection of a pseudo-Riemannian metric, then $(N,\nabla)$ is affine Osserman implies $(N,g)$ is
Osserman; the converse implication is false in general as not every Osserman manifold is nilpotent
Osserman. If $x=(x_1,...,x_m)$ are local coordinates on $N$, the Christoffel symbols are given by
$\nabla_{\partial_{x_i}}\partial_{x_j}=\sum_k\Gamma_{ij}{}^k\partial_{x_k}$.

Let $\mathcal{A}:=(N,\nabla)$ be a $2$-dimensional affine manifold.
 Let
$(x_3,x_4)$ be local coordinates on $N$.  Let $\omega=x_1dx_3+x_2dx_4\in T^*N$; $(x_1,x_2)$ are
the dual fiber coordinates. Let $\xi=\xi_{ij}(x_3,x_4)\in C^\infty(S^2(T^*N))$ be an auxiliary
symmetric bilinear form. The \emph{deformed Riemannian extension} is the Walker metric on $T^*N$
defined by setting
\begin{equation}\label{eqn-1.c}
\begin{array}{l}
g(\partial_{x_1},\partial_{x_3})=g(\partial_{x_2},\partial_{x_4})=1,\\
g(\partial_{x_3},\partial_{x_3})=-2x_1\Gamma_{33}{}^3(x_3,x_4)
-2x_2\Gamma_{33}{}^4(x_3,x_4)+\xi_{33}(x_3,x_4),\\
g(\partial_{x_3},\partial_{x_4})=-2x_1\Gamma_{34}{}^3(x_3,x_4)
-2x_2\Gamma_{34}{}^4(x_3,x_4)+\xi_{34}(x_3,x_4),\\
g(\partial_{x_4},\partial_{x_4})=-2x_1\Gamma_{44}{}^3(x_3,x_4)
-2x_2\Gamma_{44}{}^4(x_3,x_4)+\xi_{44}(x_4,x_4)\,.
\end{array}\end{equation}
The crucial fact  \cite{YI} is that the resulting neutral signature pseudo-Riemannian
manifold $\mathcal{M}$ is independent of the particular coordinates $(x_3,x_4)$ which were chosen
and is determined by $(N,\nabla,\xi)$. Moreover, proceeding as in \cite{GVV99} one has that $(N,\nabla)$
is affine Osserman if and only if the deformed Riemannian extension is Osserman for any choice
of $\xi$.

Assuming that $g_{33}=g_{44}=0$ on $M$ is equivalent to assuming that
\begin{equation}\label{eqn-1.x3}
\Gamma_{33}{}^3=\Gamma_{44}{}^3=\Gamma_{33}{}^4=\Gamma_{44}{}^4=\xi_{33}=\xi_{44}=0
\end{equation} on
$N$, i.e. that there exist coordinates on
$N$ where the two families of coordinate lines on $N$ are parallel and $\xi$-null.

We use the correspondence between $\mathcal{A}$ and
$\mathcal{M}$ to express the conditions which appear in Theorems \ref{thm-1.3}, \ref{thm-1.5}, and
\ref{thm-1.6} in a natural and covariant setting.

\begin{theorem}\label{thm-1.8}
Let $\mathcal{A}$ be a 2-dimensional affine manifold satisfying Equation (\ref{eqn-1.x3}). Let $\mathcal{M}$ be the
deformed Riemannian extension defined by Equation (\ref{eqn-1.c}). Decompose
$\rho_{\mathcal{A}}={\rho_{\mathcal{A}}^s}+{\rho_{\mathcal{A}}^a}$ into the symmetric and the
anti-symmetric parts. Then
\begin{enumerate}
\item ${\rho_{\mathcal{A}}^a}=0$ $\Leftrightarrow$ $\mathcal{M}$ is  curvature--curvature commuting
$\Leftrightarrow$
$\mathcal{M}$ is  curvature--Ricci commuting $\Leftrightarrow$ $\mathcal{M}$ is  Jacobi--Ricci commuting.
\item ${\rho_{\mathcal{A}}^s}=0$ $\Leftrightarrow$ $\rho_{\mathcal{A}}=0$ $\Leftrightarrow$ $\mathcal{A}$
is affine Osserman
$\Leftrightarrow$ $\mathcal{M}$ is Osserman $\Leftrightarrow$ $\mathcal{M}$ is  curvature--Jacobi commuting
$\Leftrightarrow$
$\mathcal{M}$ is  Jacobi--Jacobi commuting.
\end{enumerate}
\end{theorem}

If $g_{33}=g_{44}=0$, then ${\rho_{\mathcal{A}}^s}=0\Leftrightarrow\rho_{\mathcal{A}}=0$. This is, of
course, a reflection of the equivalence of conditions (1) and (2) in Lemma \ref{lem-1.4}; so far we have
only been considering two different conditions on $g_{34}$. However, this is not
the case for a more general affine extension. The following is the analogue of Theorem \ref{thm-1.8}
in the more general context; in contrast to the situation with Theorem \ref{thm-1.8}, there are 4 cases of
interest and not just 2. The following result extends Theorem \ref{thm-1.8} to the general covariant setting of
the cotangent bundle of a $2$-dimensional manifold:

\begin{theorem}\label{thm-1.10}
Let $\mathcal{A}$ be a 2-dimensional affine manifold and let $\mathcal{M}$ be the deformed
Riemannian extension defined by Equation (\ref{eqn-1.c}); we impose no additional restrictions on $\nabla$.
Then:\begin{enumerate}
\item ${\rho_{\mathcal{A}}^a}=0$ $\Leftrightarrow$ $\mathcal{M}$ is  curvature--curvature commuting.
\item ${\rho_{\mathcal{A}}^s}=0$ $\Leftrightarrow$ $\mathcal{A}$ is affine Osserman $\Leftrightarrow$
$\mathcal{M}$ is Osserman.
\item ${\rho_{\mathcal{A}}^a}=0$ or ${\rho_{\mathcal{A}}^s}=0$ $\Leftrightarrow$ $\mathcal{M}$ is
curvature--Ricci commuting $\Leftrightarrow$ $\mathcal{M}$ is  Jacobi--Ricci commuting.
\item $\rho_{\mathcal{A}}=0$ $\Leftrightarrow$ $\mathcal{M}$ is  curvature--Jacobi commuting
$\Leftrightarrow$
$\mathcal{M}$ is  Jacobi--Jacobi commuting.
\end{enumerate}
\end{theorem}

\begin{remark}\rm
If $\nabla$ is the torsion free connection on $\mathbb{R}^2$ with non-zero Christoffel symbols
$
    \nabla_{\partial_{x_3}}\partial_{x_4}=\nabla_{\partial_{x_4}}\partial_{x_3} = f(x_3)\partial_{x_3}$, 
    $\nabla_{\partial_{x_4}}\partial_{x_4}= f(x_3) \partial_{x_4}$,
for $f=f(x_3)$ with $\dot f(x_3)\neq 0$, we have ${\rho_{\mathcal{A}}^s}=0$ while
${\rho_{\mathcal{A}}^a}\neq 0$. Moreover, for the choice
$
    \nabla_{\partial_{x_3}}\partial_{x_3}=f(x_3,x_4) \partial_{x_4}$,
for $f=f(x_3,x_4)$ with $f_{/4}\neq 0$, it follows that ${\rho_{\mathcal{A}}^s}\neq 0$
while ${\rho_{\mathcal{A}}^a}=0$. Thus, in contrast to the case studied in Theorem \ref{thm-1.8},
these conditions are distinct. Combined with our previous results, this shows the four possibilities
in Theorem \ref{thm-1.10} are distinct.
\end{remark}

\begin{remark} \label{aaa}\rm
Adopt the notation of Equation (\ref{eqn-1.b}); one now has that any of the conditions Osserman,
Einstein,  curvature--curvature commuting, curvature--Jacobi commuting,  Jacobi--Jacobi commuting,
curvature--Ricci commuting, or  Jacobi--Ricci commuting implies
that $g_{34}$ is affine in the variables $\{x_1,x_2\}$ and hence $\mathcal{M}$ is a Riemannian
extension. This is not the case, however, in the more general context of Equation (\ref{eqn-1.a}).
Indeed, let $\mathcal{M}$ have the form given in Equation (\ref{eqn-1.a}) where
 \begin{enumerate}
\item $g_{33}=4kx_1^2-\textstyle\frac1{4k}f(x_4)^2$,
$g_{44}=4kx_2^2$, and $g_{34}=4kx_1x_2+x_2f(x_4)-\frac1{4k}\dot
f(x_4)$ for $f=f(x_4)$ non-constant and for $k\neq 0$. Then
$\mathcal{M}$ is Osserman with eigenvalues $\{0,4k,k,k\}$ and $\rho\ne0$. The
Jacobi operators are diagonalizable at $P$ $\Leftrightarrow$
$24kf(x_4)\dot f(x_4)x_2-12k\ddot f(x_4)x_1
+3f(x_4)\ddot f(x_4)+4\dot f(x_4)^2=0$, $\mathcal{M}$ is  Jacobi--Ricci commuting and  curvature--Ricci commuting,
$\mathcal{M}$ is neither  Jacobi--Jacobi commuting nor curvature--Jacobi commuting nor  curvature--curvature commuting.
\item $g_{33}=x_1x_2$, $g_{44}=-x_1x_2$, and $g_{34}=(x_2^2-x_1^2)/2$. Then $\mathcal{M}$ is
 curvature--curvature commuting,  curvature--Ricci commuting,  Jacobi--Ricci commuting, and $\rho^2=-\operatorname{id}$. However
$\mathcal{M}$ is not Einstein nor  curvature--Jacobi commuting nor  Jacobi--Jacobi commuting.
\end{enumerate}
(We refer to \cite{DGV06} for the proof of Assertion (1) and to \cite{GN06} for the proof of
Assertion (2)).
\end{remark}

\smallskip

Here is a brief outline to the paper.  In Section \ref{sect-3}, we reduce the proof of Theorems
\ref{thm-1.3}, \ref{thm-1.5}, and \ref{thm-1.6} to the case where $g_{34}$ is affine in
$\{x_1,x_2\}$. In Section \ref{sect-4}, we study the Osserman condition to establish Theorem
\ref{thm-1.3}.  We then turn to the study of commutativity conditions. In Section \ref{sect-5}, we establish Theorem
\ref{thm-1.5} and in Section
\ref{sect-6}, we verify Theorem \ref{thm-1.6}. Section
\ref{sect-7} deals with affine extensions and the proof of Theorem \ref{thm-1.8}. Finally, in Appendix
\ref{sect-2}, we prove the technical result stated in  Lemma \ref{lem-1.4}. We shall omit the proof of Theorem \ref{thm-1.10}
as the proof is similar to the proof we shall give to establish Theorem \ref{thm-1.8}; details are
available from the authors upon request.

\section{Reduction to an affine warping function}\label{sect-3}
Let $\mathcal{M}$ be given by Equation (\ref{eqn-1.b}). One of the crucial features we shall exploit is that
$\rho$, $\mathcal{J}(x)$, and $\mathcal{R}(x)$ are polynomial in the jets of $g_{34}$. One has by \cite{BGVa05}
that the non-zero components of the curvature tensor are, after adjusting for a difference in the sign
convention used therein, given by:
$$
\begin{array}{ll}
R_{1334}=-\frac14(g_{34/1}g_{34/2}-2g_{34/13}),&
R_{1314}=\frac12g_{34/11},\\
R_{1434}=-\frac14(-g_{34/1}^2+2g_{34/14}),&
R_{1324}=\frac12g_{34/12},\vphantom{\vrule height 11pt}\\
R_{2334}=-\frac14(g_{34/2}^2-2g_{34/23}),&
R_{1423}=\frac12g_{34/12},\vphantom{\vrule height 11pt}\\
R_{2434}=-\frac14(-g_{34/1}g_{34/2}+2g_{34/24}),&
R_{2324}=\frac12g_{34/22},\vphantom{\vrule height 11pt}\\
R_{3434}=-\frac12(-g_{34}g_{34/1}g_{34/2}+2g_{34/34})\,.\vphantom{\vrule height 11pt}
\end{array}$$
One can now use the metric to raise indices and compute $\mathcal{J}$, $\mathcal{R}$, and $\rho$.

If $P$ is a polynomial and
if $U$ is a monomial expression, we let $c(P,U)$ be the coefficient of $U$ in $P$. Let
\begin{equation}\label{eqn-3.a}
\begin{array}{ll}
\mathcal{R}_1:=\mathcal{R}(\textstyle\sum_ia_i\partial_{x_i},\textstyle\sum_jb_j\partial_{x_j}),&
\mathcal{R}_2:=\mathcal{R}(\textstyle\sum_ic_i\partial_{x_i},\textstyle\sum_jd_j\partial_{x_j}),\\\
\mathcal{J}_1:=\mathcal{J}(\textstyle\sum_iv_i\partial_{x_i}),&
  \mathcal{J}_2:=\mathcal{J}(\textstyle\sum_iw_i\partial_{x_i})\,.\vphantom{\vrule height 11pt}
\end{array}\end{equation}
We used Mathematica to assist us in the following computations. One has
\begin{eqnarray*}
&&\rho_{21}=\textstyle\frac12g_{34/11},\qquad
  \rho_{12}=\textstyle\frac12g_{34/22},\\
&&c(\{\rho\mathcal{R}_1-\mathcal{R}_1\rho\}_{21},a_4b_1)=-\textstyle\frac14g_{34/11}^2,\\
&&c(\{\rho\mathcal{R}_1-\mathcal{R}_1\rho\}_{12},a_3b_2)=-\textstyle\frac14g_{34/22}^2,\\
&&c(\{\rho\mathcal{J}_1-\mathcal{J}_1\rho\}_{21},v_1v_4)=\textstyle\frac14g_{34/11}^2,\\
&&c(\{\rho\mathcal{J}_1-\mathcal{J}_1\rho\}_{12},v_2v_3)=\textstyle\frac14g_{34/22}^2,\\
&&c(\{\mathcal{R}_1\mathcal{R}_2-\mathcal{R}_2\mathcal{R}_1\}_{21},a_4b_1c_3d_1)=-\textstyle\frac14g_{34/11}^2,\\
&&c(\{\mathcal{R}_1\mathcal{R}_2-\mathcal{R}_2\mathcal{R}_1\}_{12},a_4b_2c_3d_2)=\textstyle\frac14g_{34/22}^2\,.
\end{eqnarray*}
Consequently, if $\mathcal{M}$ is Einstein or  curvature--Ricci commuting or  Jacobi--Ricci commuting  or  curvature--curvature commuting, we
have $g_{34/11}=g_{34/22}=0$ so
$$g_{34}=x_1p(x_3,x_4)+x_2q(x_3,x_4)+x_1x_2r(x_3,x_4)+s(x_3,x_4)\,.$$
We then compute:
\begin{eqnarray*}
&&c(\rho_{13},x_1^2)=-\textstyle\frac12r(x_3,x_4)^2,\\
&&c(\{\rho\mathcal{R}_1-\mathcal{R}_1\rho\}_{13},a_4b_3x_1^3x_2)=-\textstyle\frac14r(x_3,x_4)^4,\\
&&c(\{\rho\mathcal{J}_1-\mathcal{J}_1\rho\}_{13},v_3v_4x_1^3x_2)=\textstyle\frac12r(x_3,x_4)^4,\\
&&c(\{\mathcal{R}_1\mathcal{R}_2-\mathcal{R}_2\mathcal{R}_1\}_{13},a_1b_4c_4d_3x_1^2x_2^2)=
\textstyle-\frac14r(x_3,x_4)^4\,.
\end{eqnarray*}
Consequently, if $\mathcal{M}$ is Einstein or  curvature--Ricci commuting or Jacobi--Ricci commuting  or  curvature--curvature commuting, we
have $g_{34/12}=0$ so $g_{34}$ is affine in $\{x_1,x_2\}$ and has the form:
\begin{equation}\label{eqn-3.b}
g_{34}=x_1p(x_3,x_4)+x_2q(x_3,x_4)+s(x_3,x_4)\,.
\end{equation}

\section{The Osserman condition.\\The proof of Theorem \ref{thm-1.3}}\label{sect-4}
Assertions (1) and (2) of Theorem \ref{thm-1.3} follow from work of \cite{BGVa05}. It is
immediate that (3a) implies (3b). Suppose (3b) holds so
$\mathcal{M}$ is Einstein and $g_{34}$ has the form of Equation (\ref{eqn-3.b}). One computes that
$$\rho=\left(\begin{array}{rrrr}
0&0&-\textstyle\frac12q^2+q_{/3}&\textstyle\frac12(pq-q_{/4}-p_{/3})\\
0&0&\textstyle\frac12(pq-q_{/4}-p_{/3})&-\textstyle\frac12p^2+p_{/4}\\
0&0&0&0\\0&0&0&0\end{array}\right)\,.$$ The equivalence of Assertions (3b), (3c), and (3d) in
Theorem \ref{thm-1.3} now follows. Suppose any of these holds. By Assertion (1), $\mathcal{M}$ is
conformally Osserman. Since $\rho=0$, $\mathcal{M}$ is Osserman.
\hfill\qedbox

\section{The proof of Theorem \ref{thm-1.5}}\label{sect-5}

A direct computation shows that if $g_{34}$ has the form given in (3) of Theorem~\ref{thm-1.5},
then $\mathcal{M}$ is both  curvature--Ricci commuting and  Jacobi--Ricci commuting. Furthermore, $\mathcal{R}\rho$ and
$\mathcal{J}\rho$ are generically non-zero. We adopt the notation of Equation (\ref{eqn-3.a}) and
(\ref{eqn-3.b}).
\begin{eqnarray*}
&&c(\{\rho\mathcal{J}_1-\mathcal{J}_1\rho\}_{14},v_4^2)=
  \textstyle\frac18(p^2-2p_{/4})(q_{/4}-p_{/3}),\\
&&c(\{\rho\mathcal{J}_1-\mathcal{J}_1\rho\}_{14},v_3^2)=\textstyle\frac18(q^2-2q_{/3})(q_{/4}-p_{/3}),\\
&&c(\{\rho\mathcal{J}_1-\mathcal{J}_1\rho\}_{14},v_3v_4)=
  -\textstyle\frac14(pq-q_{/4}-p_{/3})(q_{/4}-p_{/3}),\\
&&c(\{\rho\mathcal{R}_1-\mathcal{R}_1\rho\}_{13},a_3b_4)=
   -\textstyle\frac14(q^2-2q_{/3})(q_{/4}-p_{/3}),\\
&&c(\{\rho\mathcal{R}_1-\mathcal{R}_1\rho\}_{14},a_3b_4)=
   \textstyle\frac14(pq-q_{/4}-p_{/3})(q_{/4}-p_{/3}),\\
&&c(\{\rho\mathcal{R}_1-\mathcal{R}_1\rho\}_{24},a_3b_4)=
   -\textstyle\frac14(p^2-2p_{/4})(q_{/4}-p_{/3})\,.
\end{eqnarray*}
Suppose either that $\mathcal{M}$ is  Jacobi--Ricci commuting or that $\mathcal{M}$ is  curvature--Ricci commuting. We assume
that $q_{/4}\ne p_{/3}$ and argue for a contradiction. The relations given above then show
$p^2=2p_{/4}$, $q^2=2q_{/3}$, and $pq=q_{/4}+p_{/3}$. We now use Lemma \ref{lem-1.4} in a crucial fashion to
see that this implies
$q_{/4}=p_{/3}=\frac12pq$ which is contrary to our assumption. \hfill\qedbox

\section{The proof of Theorem \ref{thm-1.6}}\label{sect-6}
We begin by studying  curvature--curvature commuting
manifolds. A direct computation shows that if $g_{34}$ has the form given in Theorem \ref{thm-1.6}
(1b) then  $\mathcal{M}$ is  curvature--curvature commuting.
Suppose conversely that $\mathcal{M}$ is curvature--curvature commuting. Again, we adopt the
notation of Equation (\ref{eqn-3.a}) and (\ref{eqn-3.b}). We compute:
\begin{eqnarray*}
&&c(\{\mathcal{R}_1\mathcal{R}_2-\mathcal{R}_2\mathcal{R}_1\}_{14},a_3b_1c_4d_3)=
\phantom{-}\textstyle\frac18(pq-2p_{/3})(q_{/4}-p_{/3}),\\
&&c(\{\mathcal{R}_1\mathcal{R}_2-\mathcal{R}_2\mathcal{R}_1\}_{23},a_4b_2c_3d_4)=
-\textstyle\frac18(pq-2q_{/4})(q_{/4}-p_{/3})\,.
\end{eqnarray*}
This implies $p_{/3}=q_{/4}$.

Next, we study  curvature--Jacobi commuting and  Jacobi--Jacobi commuting manifolds. We polarize $\mathcal{J}$ to define
$\mathcal{J}(x,y)z:=\textstyle\frac12\{\mathcal{R}(z,x)y+\mathcal{R}(z,y)x\}$. If Assertion (2a)
holds in Theorem \ref{thm-1.6}, then
$\mathcal{J}\mathcal{R}=\mathcal{R}\mathcal{J}=\mathcal{J}^2=\mathcal{R}^2=0$ and $\mathcal{M}$ is
curvature--Jacobi commuting and  Jacobi--Jacobi commuting (see   \cite{BGb06}). Suppose that $\mathcal{M}$ is
curvature--Jacobi commuting. Then
$$0=\mathcal{R}(x,y)\mathcal{J}(x)x=\mathcal{J}(x)\mathcal{R}(x,y)x=-\mathcal{J}(x)^2y$$
and $\mathcal{J}(x)^2=0$ for all $x$. If $\mathcal{M}$ is  Jacobi--Jacobi commuting, then
$$0=\mathcal{J}(x,y)\mathcal{J}(x)x=\mathcal{J}(x)\mathcal{J}(x,y)x=-\textstyle\frac12\mathcal{J}(x)^2y$$
and again $\mathcal{J}(x)^2=0$ for all $x$. Thus either Assertion (2b) or Assertion (2c) of
Theorem \ref{thm-1.6} implies Assertion (2d) of Theorem \ref{thm-1.6} holds. If $\mathcal{J}(x)^2=0$,
then $\rho=0$. Finally, if $\rho=0$, we can use Theorem \ref{thm-1.3} (3) and once again Lemma \ref{lem-1.4} is
used to see that
$g_{34}=x_1p(x_3,x_4)+x_2q(x_3,x_4)+s(x_3,x_4)$ where $p^2=2p_{/4}$, $q^2=2q_{/3}$, and
$p_{/3}=q_{/4}=\frac12pq$. We compute that:
$$
\mathcal{R}_1=\left(\begin{array}{llll}
0&0&0&*\\0&0&*&0\\
0&0&0&0\\
0&0&0&0\end{array}\right)\quad\text{and}\quad
\mathcal{J}_1=\left(\begin{array}{llll}
0&0&*&*\\
0&0&*&*\\
0&0&0&0\\
0&0&0&0\end{array}\right)\,.
$$
This verifies that Theorem \ref{thm-1.6} (2a)   holds; the equivalence of (2e) and (2f) is
provided by Theorem \ref{thm-1.3}.\hfill\qedbox

\section{Affine extensions -- the proof of Theorem \ref{thm-1.8}}\label{sect-7}
Theorem \ref{thm-1.8} will follow from Theorems \ref{thm-1.3}-\ref{thm-1.6} and from the following result:
\begin{lemma}\label{lem-7.1}
Let $\mathcal{A}$ be as in Theorem \ref{thm-1.8}. Let $p=-2\,\Gamma_{34}{}^3$ and
$q=-2\,\Gamma_{34}{}^4$.
\begin{enumerate}
\item ${\rho_{\mathcal{A}}^a}=0$ $\Leftrightarrow$ $p_{/3}=q_{/4}$.
\item ${\rho_{\mathcal{A}}^s}=0$ $\Leftrightarrow$ $\rho_{\mathcal{A}}=0$ $\Leftrightarrow$
$\rho_{\mathcal{M}}=0$.
\end{enumerate}
\end{lemma}

\begin{proof}Let $\nabla$ be a torsion free connection on $\mathbb{R}^2$ with non-zero Christoffel symbols
$$\nabla_{\partial_{x_3}}\partial_{x_4}=\nabla_{\partial_{x_4}}\partial_{x_3}
 =-\textstyle\frac12p(x_3,x_4)\partial_{x_3}-\textstyle\frac12q(x_3,x_4)\partial_{x_4}\,.$$
We compute:
\begin{eqnarray*}
&&\mathcal{R}_{\mathcal{A}}(\partial_{x_3},\partial_{x_4})\partial_{x_3}
  =\textstyle\nabla_{\partial_{x_3}}(-\frac12p\partial_{x_3}-\frac12q\partial_{x_4})\\
&&\qquad=\textstyle
  -\frac12p_{/3}\partial_{x_3}-\frac12q_{/3}\partial_{x_4}-\frac12q\nabla_{\partial_{x_3}}\partial_{x_4}\\
&&\qquad=\textstyle(\frac14pq-\frac12p_{/3})\partial_{x_3}+(\frac14q^2-\frac12q_{/3})\partial_{x_4},\\
&&\mathcal{R}_{\mathcal{A}}(\partial_{x_3},\partial_{x_4})\partial_{x_4}
   =\textstyle(\frac12p_{/4}-\frac14p^2)\partial_{x_3}+(\frac12q_{/4}-\frac14pq)\partial_{x_4}\,.
\end{eqnarray*}
The Ricci tensor is then given by:
$$\begin{array}{ll}
\rho_{\mathcal{A}}(\partial_{x_3},\partial_{x_3})=\frac12q_{/3}-\frac14q^2,&
\rho_{\mathcal{A}}(\partial_{x_3},\partial_{x_4})=\frac14pq-\frac12q_{/4},\\
\rho_{\mathcal{A}}(\partial_{x_4},\partial_{x_3})=\frac14pq-\frac12p_{/3},&
\rho_{\mathcal{A}}(\partial_{x_4},\partial_{x_4})=\frac12p_{/4}-\frac14p^2\,.\vphantom{\vrule height
11pt}
\end{array}$$
It now follows that $\rho_{\mathcal{A}}$ is symmetric $\Leftrightarrow$ $q_{/4}=p_{/3}$;
$\rho_{\mathcal{A}}$ is anti-symmetric $\Leftrightarrow$ $2q_{/3}=q^2$, $2p_{/4}=p^2$, and
$pq=q_{/4}+p_{/3}$. Finally $\rho_{\mathcal{A}}=0$ $\Leftrightarrow$ $q^2=2q_{/3}$, $p^2=2p_{/4}$, and
$p_{/3}=q_{/4}=\frac12pq$. Lemma \ref{lem-7.1} now follows from Lemma \ref{lem-1.4}.\end{proof}

\appendix
\section{A technical Lemma in PDE's -- the proof of Lemma \ref{lem-1.4}}\label{sect-2}
If $0\ne(a_0,a_3,a_4)$, set
\begin{equation}\label{eqn-2.a}
p:=-2a_4(a_0+a_3x_3+a_4x_4)^{-1}\quad\text{and}\quad q=-2a_3(a_0+a_3x_3+a_4x_4)^{-1}\,.
\end{equation} We note
that if $\lambda\ne0$, then $(\lambda a_0,\lambda a_3,\lambda a_4)$ and $(a_0,a_3,a_4)$ determine
the same functions $p$ and $q$ in Equation (\ref{eqn-2.a}). Thus we may regard $(a_0,a_3,a_4)$ as
belonging to the real projective space
$\mathbb{RP}^2:=\{\mathbb{R}^3-\{0\}\}/\{\mathbb{R}-\{0\}\}$. If $a_4=0$, then $p=0$; if $a_3=0$,
then $q=0$.

Clearly Condition (3) of Lemma \ref{lem-1.4} implies Condition (2) of Lemma \ref{lem-1.4} and
Condition (2) of Lemma \ref{lem-1.4} implies Condition (1) of Lemma \ref{lem-1.4}. Thus we must
show that if $p^2=2p_{/4}$, if $q^2=2q_{/3}$, and if $pq=p_{/3}+q_{/4}$, then $p$ and $q$ have the
form given in Equation (\ref{eqn-2.a}). Set
\begin{eqnarray*}
&&\mathcal{O}_p:=\{(x_1,x_2,x_3,x_4)\in\mathcal{O}:p(x_3,x_4)\ne0\},\\
&&\mathcal{O}_q:=\{(x_1,x_2,x_3,x_4)\in\mathcal{O}:q(x_3,x_4)\ne0\}\,.
\end{eqnarray*}
We suppose first that $\mathcal{O}_p\cap\mathcal{O}_q$ is non-empty. Let $B$ be a closed ball in
$\mathbb{R}^4$ with non-empty interior which is contained in $\mathcal{O}$ and which has
$\operatorname{int}(B)\subset\mathcal{O}_p\cap\mathcal{O}_q$. We integrate the equation
$p^2=2p_{/4}$ on $\operatorname{int}(B)$ to express
\begin{equation}\label{eqn-2.b}
p(x_3,x_4)=-2(\xi(x_3)+x_4)^{-1}\quad\text{on}\quad\operatorname{int}(B)\,.
\end{equation}
We use the relation $pq=p_{/3}+q_{/4}$ to conclude
$$-2(\xi(x_3)+x_4)^{-1}q=2\dot\xi(x_3)(\xi(x_3)+x_4)^{-2}+q_{/4}(x_3,x_4)\,.$$
This relation can be written in the form $\{q(x_3,x_4)(\xi(x_3)+x_4)^2\}_{/4}=-2\dot\xi(x_3)$.
Consequently
\begin{equation}\label{eqn-2.c}
q(x_3,x_4)=\{\phi(x_3)-2\dot\xi(x_3)x_4\}(\xi(x_3)+x_4)^{-2}\,.
\end{equation}
We set $q^2=2q_{/3}$ and clear denominators to obtain the relation:
\begin{equation}\label{eqn-2.d}
\begin{array}{l}
\{\phi(x_3)-2\dot\xi(x_3)x_4\}^2
=2\{\dot\phi(x_3)-2\ddot\xi(x_3)x_4\}(\xi(x_3)+x_4)^2\\
\phantom{\{\phi(x_3)-2\dot\xi(x_3)x_4\}^2}-
2\{\phi(x_3)-2\dot\xi(x_3)x_4\}2\dot\xi(x_3)(\xi(x_3)+x_4)\,.
\end{array}\end{equation}
Setting the coefficient of $x_4^3$ equal to zero then yields $\ddot\xi(x_3)=0$ so
$\xi(x_3)=\alpha_0+\alpha_1x_3$ and Equation (\ref{eqn-2.d}) becomes:
\begin{equation}\label{eqn-2.e}
\begin{array}{l}
\{\phi(x_3)-2\alpha_1x_4\}^2=2\dot\phi(x_3)(\alpha_0+\alpha_1x_3+x_4)^2\\
\phantom{\{\phi(x_3)-2\alpha_1x_4\}^2}
-4\alpha_1(\phi(x_3)-2\alpha_1x_4)(\alpha_0+\alpha_1x_3+x_4)\,.
\end{array}\end{equation}
Examining the coefficient of $x_4^2$ in Equation (\ref{eqn-2.e}) shows that
$\dot\phi(x_3)=-2\alpha_1^2$ so $\phi(x_3)=\beta_0-2\alpha_1^2x_3$. Equation (\ref{eqn-2.e}) then
further simplifies to become:
\begin{equation}\label{eqn-2.f}
\begin{array}{l}
(\beta_0-2\alpha_1^2x_3-2\alpha_1x_4)^2
=-4\alpha_1^2(\alpha_0+\alpha_1x_3+x_4)^2\\
\phantom{(\beta_0-2\alpha_1^2x_3-2\alpha_1x_4)^2}
-4\alpha_1(\beta_0-2\alpha_1^2x_3-2\alpha_1x_4)(\alpha_0+\alpha_1x_3+x_4)\,.
\end{array}\end{equation}
This leads to the relation $\beta_0^2=-4\alpha_1^2\alpha_0^2-4\beta_0\alpha_1\alpha_0$ which
implies $\beta_0=-2\alpha_1\alpha_0$. Equations (\ref{eqn-2.b}) and (\ref{eqn-2.c}) now yield
\begin{equation}\label{eqn-2.g}
\begin{array}{l}
p(x_3,x_4)=-2(\alpha_0+\alpha_1x_3+x_4)^{-1},\\
q(x_3,x_4)=-2\alpha_1(\alpha_0+\alpha_1x_3+x_4)^{-1}\,.\vphantom{\vrule height 11pt}
\end{array}\end{equation}
By continuity, Equations (\ref{eqn-2.g}) hold  on the closed ball $B$  and in particular $p$ and
$q$ do not vanish on $B$. It now follows that $\mathcal{O}=\mathcal{O}_p=\mathcal{O}_q$. Analytic
continuation now shows $p$ and $q$ are given by Equation (\ref{eqn-2.g}) on all of $\mathcal{O}$
and thus Assertion (3) holds.

We therefore assume $\mathcal{O}_p\cap\mathcal{O}_q$ is empty. If $\mathcal{O}_p$ and
$\mathcal{O}_q$ are both empty, then $p=q=0$ and we may take $(a_0,a_3,a_4)=(1,0,0)$ to obtain a
representation of the form given in (3). We therefore assume $\mathcal{O}_q$ is non-empty; the
case $\mathcal{O}_p$ is non-empty is handled similarly. Let $B$ be a closed ball in $\mathbb{R}^4$
with non-empty interior which is contained in $\mathcal{O}$ and which satisfies
$\operatorname{int}(B)\subset\mathcal{O}_q$. We integrate the equation $q^2=2q_{/3}$ to express
$$q=-2(x_3+\eta(x_4))^{-1}\quad\text{on}\quad\operatorname{int}(B)\,.$$
Since $pq=p_{/3}+q_{/4}$ and since $p=0$ on $\operatorname{int}(B)$, we have $\dot\eta=0$ and
hence
\begin{equation}\label{eqn-2.h}
q=-2(x_3+a)^{-1}\quad\text{on}\quad\operatorname{int}(B)\,.
\end{equation}
Again, by continuity, this representation holds on all of $B$ and thus $q$ is non-zero on $B$.
Thus $\mathcal{O}=\mathcal{O}_q$, Equation (\ref{eqn-2.h}) holds on all of $\mathcal{O}$, and
$p=0$ on all of $\mathcal{O}$. This again obtains a representation for $p$ and $q$ of the form
given in Assertion (3).
  \hfill\qedbox

\section*{Acknowledgments} Research of M. Brozos-V\'azquez and of P. Gilkey
was partially supported by the Max Planck Institute for Mathematics in the Sciences (Germany)  and
by Project MTM2006-01432 (Spain). Research of E. Garc\'{\i}a--R\'{\i}o and of R. V\'azquez-Lorenzo was
partially supported by PGIDIT06PXIB207054PR (Spain).  We are grateful for helpful suggestions from the referees.

\end{document}